\documentclass[]{amsart}

\def\bn{\mathbb{N}}
\def\br{\mathbb{R}} 
\def\bc{\mathbb{C}} 
\def\bj{\mathbb{J}}
\def\h{\mathcal{H}} 

\newcommand\ca{\mathcal{A}}

\newcommand\cz{\mathcal{Z}}
\newcommand\crr{\mathcal{R}}
\newcommand{\bh}{{\rm B}(\mathcal{H})}
\newcommand{\kh}{{\rm K}(\mathcal{H})}
\newcommand{\er}{{\rm E}(\crr)}
\newcommand{\ea}{{\rm E}(A)}

\newcommand{\oz}{\omega_{\cz}}

\newcommand{\ozp}{\omega_{\cz}^+}
\newcommand{\ozm}{\omega_{\cz}^-}
\newcommand{\rzp}{\rho_{\cz}^+}
\newcommand{\rzm}{\rho_{\cz}^-}
\newcommand{\pdr}{\crr_{\sharp}}

\newcommand{\tom}{\tilde{\omega}}
\newcommand{\tr}{\tilde{\rho}}
\newcommand{\bhp}{{\rm B}(\mathcal{H}_{\pi})}
\newcommand{\normc}[1]{\overline{\overline{#1}}}

\newtheorem{theorem}{Theorem}[section]
\newtheorem{lemma}[theorem]{Lemma}
\newtheorem{co}[theorem]{Corollary}

\theoremstyle{remark}
\newtheorem{re}[theorem]{Remark}
\theoremstyle{definition}

\newtheorem{ex}[theorem]{Example}
\numberwithin{equation}{section}

\makeatletter
\@namedef{subjclassname@2020}{\textup{2020} Mathematics Subject Classification}
\makeatother

\begin{document}

\title[]{Which states can be reached from a given state by  unital completely positive maps?} 

\author{Bojan Magajna} 
\address{Department of Mathematics\\ University of Ljubljana\\
Jadranska 21\\ Ljubljana 1000\\ Slovenia}
\email{Bojan.Magajna@fmf.uni-lj.si}

\thanks{The author acknowledges the financial support from the
Slovenian Research Agency (research core funding no. P1-0288).}



\begin{abstract}For a state $\omega$ on a C$^*$-algebra $A$ we characterize all states $\rho$ in the weak* closure of the set of all states of the form $\omega\circ\varphi$, where $\varphi$ is a map on $A$ of the form $\varphi(x)=\sum_{i=1}^na_i^*xa_i,$ $\sum_{i=1}^na_i^*a_i=1$ ($a_i\in A$, $n\in\bn$). These are precisely the states $\rho$  that satisfy $\|\rho|J\|\leq\|\omega|J\|$ for each ideal $J$ of $A$.  The corresponding question for normal states on a von Neumann algebra $\crr$ (with the weak* closure replaced by the norm closure) is also considered. All normal states of the form $\omega\circ\psi$, where $\psi$ is a quantum channel on $\crr$ (that is, a map of the form $\psi(x)=\sum_ja_j^*xa_j$, where $a_j\in\crr$ are such that the sum $\sum_ja_j^*a_j$ converge to $1$ in the weak operator topology) are characterized. A variant of this topic for  hermitian functionals instead of states is investigated. Maximally mixed states are shown to vanish on the strong radical of a C$^*$-algebra and for properly infinite von Neumann algebras the converse also holds.
\end{abstract}

\maketitle

\section{Introduction}

For two  states $\omega$ and $\rho$ on a C$^*$-algebra $\crr$,  $\rho$ is regarded to be {\em unitarily more mixed} than $\omega$ if $\rho$ is contained in the weak* closure of the convex hull of the unitary orbit of $\omega$. In \cite{A}, \cite{AU1} and \cite{W} Alberti, Uhlmann and Wehrl studied the notion of maximally unitarily mixed states on von Neumann algebras and such states were characterized by Alberti in \cite{A}. Recently this topic has been revitalized in the broader context of C$^*$-algebras by Archbold, Robert and Tikuisis \cite{ART}, who proved among other things that the weak*closure of the set of maximally unitarily mixed states on a C$^*$-algebra $A$ is equal to the weak* closure of the convex hull of tracial states and states that factor through simple traceless quotients of $A$. 
However, the evolution of open quantum systems is not always unitary, but is described by more general completely positive (trace preserving) maps of the form $\omega\mapsto\sum a_i\omega a_i^*$, say on the predual of $\bh$, so it seems worthwhile to study also a less restrictive  notion of when one state is more mixed than other.  The dual of such a map is a unital completely positive map of the form 
\begin{equation}\label{e}x\mapsto\sum a_i^*xa_i,\ \ \sum a_i^*a_i=1\end{equation} on $\crr=\bh$. Let   $\ea$ be the set of all unital completely positive maps on $A$ of the form (\ref{e}), where $a_i\in A$ and the sums have only finitely many terms. A natural question in this context is, when a  state $\rho$ on a C$^*$-algebra $A$ (or a normal state on a von Neumann algebra $\crr$) is in the weak* closure (or the norm closure) of the set $\omega\circ\ea$ of all states of the form
$\omega\circ\psi$, where $\omega$ is a fixed state (perhaps normal in the case of von Neumann algebras) and $\psi$ runs over the set $\ea$. In Section 2 we show  for normal states on a von Neumann algebra $\crr$ that $\rho$ is in the norm closure of $\omega\circ\er$  if and only if $\rho$ and $\omega$ agree on the center of $\crr$.  We also study  the same topic for hermitian normal functionals on $\crr$ and provide an explicit normal mapping $\psi$ in the point-weak* closure of $\er$ such that $\rho=\omega\circ\psi$. In the special case of $\crr=\bh$ hermitian normal functionals are just hermitian trace class operators and  maps mapping one such operator to another have been constructed by M-H. Hsu, D. Li-Wei Kuo and M-C. Tsai in \cite{HKT} and by Y. Li, and  H-K Du in \cite{LD}, but they do not study the question if such maps are in the closure of ${\rm E}(\bh)$. 

For a normal state $\omega$ on a von Neumann algebra $\crr\subseteq\bh$ and a map $\phi$ of the form (\ref{e}), where $a_i\in\crr$ and the sums may have infinitely many terms (that is, $\phi$ is a quantum channel) any state of the form $\rho=\omega\circ\phi$ has the following property: if $\tom$ is a normal state on $\bh$ that extends $\omega$, then there is a normal state $\tr$ on $\bh$ that extends $\rho$ such that $\tr$ and $\tom$ coincide on the commutant $\crr^{\prime}$ of $\crr$ (namely, $\tr=\tom\circ\tilde{\phi}$, where $\tilde{\phi}$ is the map on $\bh$ given by the same formula as $\phi$ on $\crr$). This property holds in any faithful normal representation of $\crr$ on a Hilbert space $\h$. In Section 2 we will see that this property characterizes states of the form $\omega\circ\phi$, where $\phi$ runs over quantum channels on $\crr$.

Then in Section 3 we study the analogous topic for hermitian functionals $\rho, \omega$ on a unital C$^*$-algebra $A$. If $A$ has Hausdorff primitive spectrum, Theorem \ref{H} shows that $\rho$ is in the weak* closure of  $\omega\circ\ea$ if and only if $\omega$ and $\rho$ agree on the center of $A$ and $\|c\rho\|\leq\|c\omega\|$ for each positive element $c$ in the center of $A$. If the primitive spectrum of $A$ is not Hausdorff, this characterization is not true any more, but an alternative one is given in Theorem \ref{H2}.

For two states $\omega$ and $\rho$ on a C$^*$-algebra $A$, $\rho$ is regarded here to be {\em more mixed than $\omega$} if $\rho$ is contained in the weak* closure $\overline{\omega\circ\ea}$ of the set $\omega\circ\ea:=\{\omega\circ\psi:\, \psi\in\ea\}$.  Then $\omega$  is called {\em maximally mixed} if for each state $\rho$ on $A$ the condition that $\rho\in\overline{\omega\circ\ea}$  implies that $\omega\in\overline{\rho\circ\ea}$; in other words, $\overline{\omega\circ\ea}$ is minimal among weak* closed $\ea$-invariant subsets of the set $S(A)$ of all states on $A$. This is a coarser relation than the one considered in references mentioned above, where instead of $\ea$ only convex combinations of unitary similarities are considered. In Section 4 we show that each maximally mixed state on a unital C$^*$-algebra $A$ must annihilate the strong radical $J_A$ of $A$ (= the intersection of all two-sided maximal ideals of $A$) and, if $A$ is a properly infinite von Neumann algebra, the converse is also true. Further, the set $S_m(A)$ of all maximally mixed states contains all states that annihilate some intersection of finitely many maximal ideals of $A$ and is therefore weak* dense in $S(A/J_A)$. These results are analogous to those of \cite{ART} and \cite{A} for unitarily maximally mixed states. For C$^*$-algebras  with the Dixmier property the authors of \cite{ART} provided a more precise determination of maximally unitarily mixed states than for general C$^*$-algebras. In our present context the role of C$^*$-algebras with the Dixmier property can be played by  weakly central C$^*$-algebras. For a weakly central C$^*$-algebra $A$ we show that the set $S_m(A)$ is weak* closed (and hence equal to the set of all states that annihilate $J_A$) if and only if each primitive ideal of $A$ which contains $J_A$ is maximal. States in $S_m(\crr)$ for a general von Neumann algebra $\crr$ are also characterized.

{\em Throughout the paper an ideal means a norm closed two-sided ideal and all C$^*$ algebras are assumed to be unital unless explicitly stated otherwise.}  

\section{The case of normal states on a von Neumann algebra}

We denote by $A^{\sharp}$ the dual of a Banach case $A$. In what follows $A$ will usually be a C$^*$-algebra. Throughout this article $\crr$ is a von Neumann algebra,  $\pdr$ its predual (that is, the space of all weak* continuous linear functionals on $\crr$) and $\cz$  the center of $R$. Basic facts concerning von Neumann algebras, that will be used here without explicitly mentioning a reference, can be found  in \cite{KR} and  \cite{T}.

We will need a preliminary result of independent interest, which in the special case (when, in the notation of Theorem \ref{th} below, $\ca=\crr$ and $\crr$ is a factor or has a separable predual, and positivity was not considered), has been proved by Chatterjee and Smith \cite{CS}. We would like to avoid the separability assumption. In its proof we will use the notion of the minimal C$^*$-tensor product over $\cz$ of two C$^*$-algebras $A$ and $B$ both containing an abelian W$^*$-algebra $\cz$ in their centers. This product $A\otimes_{\cz}B$ (\cite{B}, \cite{M1}, \cite{GM}) can be defined as the closure of the image of the algebraic tensor product $A\odot_{\cz}B$ in $\oplus_{t\in\Delta}A(t)\otimes B(t)$,
where $\Delta$ is the maximal ideal space of $\cz$ and, for each $t\in\Delta$, $A(t)$ denotes the quotient C$^*$-algebra $A/(tA)$, where $tA$ is the closed ideal in $A$ generated by $t$ (and similarly for $B(t)$).  (If at least one of the algebras $A$, $B$ is exact, which will be the case in our application in the proof of Theorem \ref{thvn}, $A\otimes_{\cz}B$ coincides with the quotient of $A\otimes B$ by the closed ideal generated by all elements of the form $az\otimes b-a\otimes zb$ ($a\in A$, $b\in B$, $z\in\cz$)  \cite[3.12]{M1}.)

\begin{theorem}\label{th}Let $\ca$ be an injective von Neumann subalgebra of a von Neumann algebra $\crr$ containing the center $\cz$ of $\crr$. Then each completely contractive
$\cz$-module map $\psi:\crr\to\ca$ is (as a map into $\crr$) in the point-weak* closure of the set consisting of all maps of the form $x\mapsto\sum_{i=1}^na_i^*xb_i$ ($x\in\crr$), where $n\in\bn$ and $a_i, b_i\in\crr$ satisfy $\sum_{i=1}^na_i^*a_i\leq1$ and $\sum_{i=1}^nb_i^*b_i\leq1$. If in addition $\psi$ is unital, then $\psi$ is in the point-weak* closure of $\er$.
\end{theorem}

\begin{proof}Let $\h$ be a Hilbert space such that $\crr\subseteq\bh$. It follows from \cite[5.5.4]{KR} that there is a natural $*$-isomorphism $\iota$ from $\crr\crr^{\prime}$ (the subalgebra of $\bh$ generated by $\crr\cup\crr^{\prime}$) onto the algebraic tensor product $\crr\odot_{\cz}\crr^{\prime}$, given by
$rr^{\prime}\mapsto r\otimes_{\cz} r^{\prime}$. By \cite[2.9]{B} the tensor norm on $\crr{\otimes}_{\cz}\crr^{\prime}$ restricted to $\crr\odot_{\cz}\crr^{\prime}$ is minimal among all C$^*$-tensor norms on $\crr\odot_{\cz}\crr^{\prime}$, hence   the $*$-homomorphism $\iota$ extends uniquely to the norm closure $\normc{\crr\crr^{\prime}}$.  Since $\ca$ is injective and commutes with $\crr^{\prime}$ the multiplication $\mu_0:\ca\otimes\crr^{\prime}\to\normc{\ca\crr^{\prime}}\subseteq\bh$ is a completely contractive $*$-homomorphism \cite[9.3.3, 3.8.5]{BO}.  But more is true: by \cite[4.2]{GM} the natural map $\ca\odot_{\cz}\crr^{\prime}\to\ca\crr^{\prime}$ extends (uniquely) to a $*$-isomorphism $\ca\otimes_{\cz}\crr^{\prime}\to \normc{\ca\crr^{\prime}}$. It follows that  the composition
$$\bh\supseteq\overline{\overline{\crr\crr^{\prime}}}\to\crr{\otimes}_{\cz}\crr^{\prime}\stackrel{\psi{\otimes}_{\cz}{\rm id}}{\longrightarrow}\mathcal{A}{\otimes}_{\cz}\crr^{\prime}\cong\normc{\mathcal{A}\crr^{\prime}}\subseteq\bh$$
is completely contractive and clearly it is an $\crr^{\prime}$-bimodule map,
hence extends to such a map $\phi$ on $\bh$ by the Wittstock extension theorem (see \cite{Wi} or \cite[3.6.2]{BLM}). By \cite{EK}   $\phi$ can be approximated in the point-weak* topology  by a net of elementary complete contractions of the form 
\begin{equation}\label{30}x\mapsto \sum_ia_i^*(k)xb_i(k)=a(k)^*xb(k)\ \  (x\in\bh)\end{equation}
where $a(k)=(a_1(k),\ldots,a_n(k))^T$ and $b(k)=(b_1(k),\ldots,b_n(k))^T$ are columns with the entries $a_i(k),b_i(k)\in\crr$ and \begin{equation}\label{31}a^*(k)a(k)=\sum_ia_i^*(k)a_i(k)\leq1,\ \ b^*(k)b(k)=\sum_ib_i^*(k)b_i(k)\leq1.\end{equation} Thus $\psi$ (=$\phi|\crr$) can also be approximated by such maps. 

Assume now in addition that $\psi$ is unital and consider a point-weak* approximation of $\psi$ of the form (\ref{30}), (\ref{31}).  Since $$0\leq(b(k)-a(k))^*(b(k)-a(k))=b(k)^*b(k)+a(k)^*a(k)-a(k)^*b(k)-b(k)^*a(k)$$$$\leq2-a(k)^*b(k)-b(k)^*a(k)\to2-2\psi(1)=0,$$
it follows that $b(k)-a(k)$ tends to $0$ in the strong operator topology. Hence $\psi$ can be approximated by maps of the form $x\mapsto a(k)^*xa(k)$ in the point-weak* operator topology. To see this, write  $$\psi(x)-a(k)^*xa(k)=(\psi(x)-a(k)^*xb(k))+(a(k)^*x(b(k)-a(k)))$$
and note that $\|a(k)^*x(b(k)-a(k))\xi\|\leq\|x\|\|(b(k)-a(k))\xi\|$ for each vector $\xi\in\h$.  Finally, as $a(k)^*a(k)$ tends to $\psi(1)=1$ in the strong operator topology, $\psi$ can be approximated by maps of the form $$x\mapsto a(k)^*xa(k)+\sqrt{1-a(k)^*a(k)}x\sqrt{1-a(k)^*a(k)},$$ that is, by unital completely positive elementary maps.
\end{proof}

\begin{lemma}\label{le00}Let $\omega$ and $\rho$ be hermitian functionals on a C$^*$-algebra $A$ such that
$\rho|Z=\omega|Z$ and $\|c\rho\|\leq\|c\omega\|$ for all $c\in Z_+$, where $Z$ is the center of $A$. Then $\rho_+|Z\leq\omega_+|Z$ and $\rho_-|Z\leq\omega_-|Z$. 

Thus, if $Z$ is a von Neumann algebra, $\omega$ and $\rho$ are normal and $p^+$ and $p^-$ are the support projections of $\omega_+|Z$ and $\omega_-|Z$, then there exists elements $c_+$ and $c_-$ in $Z$ such that $0\leq c_+\leq p^+$, $0\leq c_-\leq p^-$, 
$$\rho_+|Z=c_+\omega_+|Z,\ \  \rho_-|Z=c_-\omega_-|Z,\ \ \mbox{and}\ \ (p^+-c_+)\omega_+|Z=(p^--c_-)\omega_-|Z.$$
\end{lemma}

\begin{proof}For each $c\in Z_+$ and $\theta\in(A^{\sharp})_+$ we have that $\|c\theta\|=(c\theta)(1)=\theta(c)$ and it is also well known that for each hermitian functional $\sigma$ the equality $\|\sigma\|=\sigma_+(1)+\sigma_-(1)=\|\sigma_+\|+\|\sigma_-\|$ holds, hence 
$$\rho_+(c)+\rho_-(c)=\|c\rho\|\leq\|c\omega\|=\omega_+(c)+\omega_-(c),$$
$$\rho_+(c)-\rho_-(c)=\rho(c)=\omega(c)=\omega_+(c)-\omega_-(c).$$ Adding and subtracting these two relations we find that
$\rho_+(c)\leq\omega_+(c)$ and $\rho_-(c)\leq\omega_-(c)$ for all $c\in Z_+$.
If $Z$, $\omega$, $\rho$, $p^+$ and $p^-$ are as in the second part of the lemma, we may regard $Z$ as $L^{\infty}(\mu)$ for some positive measure $\mu$ and then the existence of elements $c_+$ and $c_-$ in $Z$ satisfying $0\leq c_+\leq p_+$, $0\leq c_-\leq p^-$ and $\rho_+|Z=c_+\omega_+|Z$, $\rho_-|Z=c_-\omega_-|Z$  follows easily, so we will verify here only the last equality in the lemma. The condition $\rho|Z=\omega|Z$ can be written as $(c_+\omega_+-c_-\omega_-)|Z=(\omega_+-\omega_-)|Z$, hence $(1-c_+)\omega_+|Z=(1-c_-)\omega_-|Z$. But $\omega_+=p^+\omega_+$ and $\omega_-=p^-\omega_-$, since $p^+$ and $p^-$ are the support projections of $\omega_+|Z$ and $\omega_-|Z$, hence the required equality follows.
\end{proof}

By \cite{Ha} or \cite{SZ} each positive functional $\omega$ on $\crr$, such that $\omega|\cz$ is weak* continuous, can be uniquely expressed as 
\begin{equation}\label{00}\omega=(\omega|\cz)\circ\omega_{\cz},\end{equation} where $\omega_{\cz}$ is a (completely) positive $\cz$-module map from $\crr$ to $\cz$ such that $\oz(1)$ is the support projection $q\in\cz$ of $\omega|{\cz}$.  If $\omega$ is weak* continuous, then so is also $\oz$. {\em Observe that the support projections of $\omega$ and $\oz$ coincide, if $\omega$ is normal.} (Indeed, for each projection $e\in\crr$ we have $0\leq\oz(e)\leq\oz(1)=q$, hence $\omega(e)=(\omega|\cz)(\oz(e))=0$ if and only if $\oz(e)=0$ since $q$ is the support projection of $\omega|\cz$.)

\begin{theorem}\label{thvn}Let $\omega,\rho$ be normal hermitian functionals on $\crr$. There exists a normal unital completely positive map $\psi:\crr\to\crr$ in the point-weak* closure of $\er$ satisfying $\psi(1)=1$ and $\psi_{\sharp}(\omega)=\rho$ if and only if 
\begin{equation}\label{10}\rho|\cz=\omega|\cz\ \ \mbox{and}\ \ \|c\rho\|\leq\|c\omega\|\ \ \forall c\in\cz_+.\end{equation}
Under this condition $\rho$ is in the norm closure of $\omega\circ\er$. 
\end{theorem}

\begin{proof}Since maps in $\er$ are unital and completely positive, they are also completely contractive. Each map in $\er$ is of the form  $\psi(x)=\sum_{i=1}^na_i^*xa_i$, where $a_i\in\crr$ and $\sum_{i=1}^na_i^*a_i=1$, hence weak* continuous and the corresponding map $\psi_{\sharp}$ on the predual $\pdr$ of $\crr$ is given by
$\psi_{\sharp}(\omega)=\sum_{i=1}^na_i\omega a_i^*$ and is a $\cz$-module map with $\|\psi_{\sharp}\|=\|\psi\|=1$. Hence $\|c\psi_{\sharp}(\omega)\|=\|\psi_{\sharp}(c\omega)\|\leq\|c\omega\|$ for each $c\in\cz$. This means that the inequality $\|(c \omega)\circ\psi\|\leq \|c\omega\|$ holds for all $\psi\in\er$, hence also for all $\psi$ in the point-weak* closure of $\er$ since $c\omega$ is weak* continuous.  If $\psi$ is a weak* continuous such map and $\rho=\psi_{\sharp}(\omega)$, then $\|c\rho\|=\|(c\omega)\circ\psi\|\leq\|c\omega\|$. Further, since $\psi|\cz={\rm id}$ for each such map $\psi$, it follows that $\rho|\cz=\omega|\cz$ for each $\rho\in\overline{\omega\circ\er}$. 

Assume now that the condition (\ref{10}) holds.  Decompose each of the functionals $\omega_+,\omega_-,\rho_+,\rho_-$ as described in (\ref{00}), so that
$$\omega=(\omega_+|\cz)\circ\ozp-(\omega_-|\cz)\circ\ozm\ \mbox{and}\ \rho=(\rho_+|\cz)\circ\rzp-(\rho_-|\cz)\circ\rzm,$$
where $\rzp,\rzm,\ozp,\ozm$ are $\cz$-module homomorphisms from $\crr$ to $\cz$ such that $p^+:=\ozp(1)$ and $p^-:=\ozm(1)$ are the support projections of $\omega_+|\cz$ and $\omega_-|\cz$. Let $p_+$ and $p_-$ be the support projections of $\omega_+$ and $\omega_-$. Observe that $p_+\leq p^+$ and $p_-\leq p^-$. (Namely, $\omega_+(1-p^+)=(\omega_+|\cz)(1-p^+)=0$ implies that $1-p^+\leq1-p_+$, hence $p_+\leq p^+$.)  By Lemma \ref{le00} there exists $c_+, c_-\in\cz$ such that  $0\leq c_+\leq p^+$, $0\leq c_-\leq p^-$,  
\begin{equation}\label{55-}\rho_+|\cz=c_+\omega_+|\cz,\ \rho_-|\cz=c_-\omega_-|\cz\ \mbox{and}\ (p^+-c_+)\omega_+|\cz=(p^--c_-)\omega_-|\cz.\end{equation} When we first tried to find a map $\psi$ satisfying the requirements of the theorem to be of the form $\psi=a\rzp+b\rzm$, where $a,b\in\crr_+$, we found that it is not always possible to simultaneously satisfy the conditions $\psi(1)=1$ and $\omega\circ\psi=\rho$ by maps of such a form. But after several attempts  we arrived to the following map: 
\begin{equation}\label{11}\psi=c_+p_+\rzp+(1-c_+p_+)(\rzm+(1-p^-)\theta).\end{equation}
Here  $\theta$ is any fixed normal positive unital $\cz$-module map from $\crr$ to $\cz$. (Such a map exists even on $\cz^{\prime}\supseteq\crr$ since $\cz^{\prime}$ is of type $I$, hence isomorphic to a direct sum of matrix algebras of the form ${\rm M}_n(\cz)$, where $n$ can be infinite.)  This map $\psi$ is positive, weak* continuous, $\cz$-module map, with the range contained in the commutative C$^*$-algebra generated by $\cz\cup\{p_+\}$, hence completely positive. We can immediately verify that $\psi$ is also unital:
$$\psi(1)=c_+p_+\rzp(1)+(1-c_+p_+)(\rzm(1)+1-p^-)=c_+p_++(1-c_+p_+)(p^-+1-p^-)=1.$$

Now we are going to compute 
\begin{equation}\label{55}\psi_{\sharp}(\omega)=\omega\circ\psi=(\omega_+|\cz)\circ\ozp\circ\psi-(\omega_-|\cz)\circ\ozm\circ\psi.\end{equation}
For this, first observe that if $f,g:\crr\to\cz$ are $\cz$-module maps and $a\in\crr$, then $(f\circ (ag))(x)=f(ag(x))=f(a)g(x)=(f(a)g)(x)$, that is, $f\circ (ag)=f(a)g$. Note also that $p^+\rzp=\rzp$ and $p^-\rzm=\rzm$ since Lemma \ref{le00} implies that the support projection of $\rho_+|\cz$ is dominated by the support projection of $\omega_+|\cz$ and similarly for $\rho_-|\cz$ and $\omega_-|\cz$. From the definition (\ref{11}) of $\psi$ and using that  $\ozp$ and $\ozm$ are $\cz$-module maps with  ranges contained in $\cz$ and mutually orthogonal support projections $p_+$ and $p_-$ (which are just the support projections of $\omega_+$ and $\omega_-$, respectively) we now compute 
\begin{equation}\label{56}\ozp\circ\psi=\ozp(c_+p_+)\rzp+\ozp(1-c_+p_+)(\rzm+(1-p^-)\theta)\end{equation}$$=c_+\rzp+(p^+-c_+)(\rzm+(1-p^-)\theta)$$
and similarly
\begin{equation}\label{57}\ozm\circ\psi=\ozm(1-c_+p_+)\rzm=p^-\rzm=\rzm.\end{equation}
From  (\ref{55}), (\ref{56}) and (\ref{57}) we have, using also (\ref{55-}) and (\ref{11}),
$$\omega\circ\psi=(\omega_+|\cz)\circ[c_+\rzp+(p^+-c_+)(\rzm+(1-p^-)\theta)]-(\omega_-|\cz)\circ\rzm$$
$$=(c_+\omega_+|\cz)\circ\rzp+[(p^+-c_+)\omega_+|\cz]\circ\rzm+[(p^+-c_+)(1-p^-)\omega_+|\cz]\circ\theta-(\omega_-|\cz)\circ\rzm$$
$$=(\rho_+|\cz)\circ\rzp+[(p^--c_-)\omega_-|\cz]\circ\rzm+[(1-p^-)(p^--c_-)p^-\omega_-|\cz]\circ\theta-(p^-\omega_-|\cz)\circ\rzm$$
$$=\rho_+-(c_-\omega_-|\cz)\circ\rzm=\rho_+-(\rho_-|\cz)\circ\rzm=\rho_+-\rho_-=\rho.$$

It follows from Theorem \ref{th} that $\psi$ is in the point-weak* closure of $\er$. Thus $\rho=\omega\circ\psi$ is in the weak closure of the convex set $\omega\circ\er$, which is the same as the norm closure by the Hahn-Banach theorem and the fact that $\crr$ is the dual of $\pdr$.
\end{proof}

When $\omega$ and $\rho$ are states, Theorem \ref{thvn} simplifies to the following corollary:

\begin{co}\label{co0}Let $\omega$ and $\rho$ be normal states on $\crr$. There exists a normal unital completely positive map $\psi$ in the point-weak* closure of $\er$ satisfying $\psi_{\sharp}(\omega)=\rho$ if and only if $\rho|\cz=\omega|\cz$. This condition is satisfied if and only if $\|c\rho\|\leq\|c\omega\|$ for all $c\in\cz_+$.
\end{co}

\begin{proof}By Theorem \ref{thvn} we only need to verify that the condition $\rho|\cz=\omega|\cz$ implies that $\|c\rho\|\leq\|c\omega\|$ for all $c\in\cz_+$ and conversely. Since $\omega$ and $\rho$ are positive, we have $\|c\rho\|=(c\rho)(1)=\rho(c)$ and $\omega(c)=\|c\omega\|$ for all $c\in\cz_+$. If $\|c\rho\|\leq\|c\omega\|$ for all $c\in\cz_+$, then $\rho(c)\leq\omega(c)$. Applying this to $1-c$ instead of $c$, where $0\leq c\leq1$, it follows that $\rho(c)=\omega(c)$ for all such $c$. But such elements span $\cz$, hence it follows that $\rho|\cz=\omega|\cz$ if and only if $\|c\rho\|\leq\|c\omega\|$ for all $c\in\cz_+$.
\end{proof}

It is well-known that on $\crr=\bh$ all normal completely positive unital maps are of the form 
\begin{equation}\label{-1}\phi(x)=\sum_{j\in \bj}a_j^*xa_j\ \ (x\in\br),
\end{equation}
where $\bj$ is some set of indexes and $a_j\in\crr$ are such that $\sum_{j\in\bj}a_j^*a_j=1$ with the convergence in the strong operator topology. Maps on $\bh$ of the form (\ref{-1}) are called {\em quantum channels} and we will use the same name for maps of such a form on a general von Neumann algebra $\crr$. It is well-known that on a general von Neumann algebra not all unital normal completely positive maps are of the form (\ref{-1}), so we still have to answer the following question:
If $\omega$ and $\rho$ are normal states on a von Neumann algebra $\crr$, when does there exist a quntum channel $\phi$ on $\crr$ such that $\omega\circ\phi=\rho?$

\begin{theorem}\label{pr-1}For normal states $\omega$ and $\rho$ on $\crr$ the following statements are equivalent:

(i) There exists a quantum channel $\phi$ on $\crr$ such that $\omega\circ\phi=\rho$.

(ii) For every faithful normal representation $\pi$ of $\crr$ on a Hilbert space $\h_{\pi}$ and any normal state $\tom$ on $\bhp$ that extends $\omega\circ\pi^{-1}$ there exists a normal state $\tr$ on $\bhp$ that extends $\rho\circ\pi^{-1}$ such that $\tom|\pi(\crr)^{\prime}=\tr|\pi(\crr)^{\prime}$.

(iii) For some faithful normal representation of $\crr$ on a Hilbert space $\h$, such that $\omega$ is the restriction to $\crr$ of a vector state $\tom$ on $\bh$, there exists  a normal state $\tr$ on $\bh$  such that $\tr|\crr=\rho$ and $\tr|\crr^{\prime}=\tom|\crr^{\prime}$.

(iv) Let $\pi_{\omega}$ be the GNS representation of $\crr$ engendered by $\omega$ on a Hilbert space $\h_{\omega}$ and let  $\xi_{\omega}$ be the corresponding cyclic vector. The state $\rho$ annihilates the kernel of $\pi_{\omega}$ and there exists a normal state $\tr$ on ${\rm B}(\h_{\omega})$ such that $\tr|\pi_{\omega}(\crr)$ is the state induced by $\rho$ on $\pi_{\omega}(\crr)\cong\crr/\ker\pi_{\omega}$ and $\tr|\pi_{\omega}(\crr)^{\prime}=\tom|\pi_{\omega}(\crr)^{\prime}$, where $\tom$ is the vector state $x\mapsto\langle x\xi_{\omega},\xi_{\omega}\rangle$ on ${\rm B}(\h_{\omega})$.
\end{theorem}

\begin{proof}(i)$\Rightarrow$(ii) If $\rho=\omega\circ\phi$, where $\phi$ is of the form (\ref{-1}), then let $\tom$ be any state on $\bhp$ extending $\omega\circ\pi^{-1}$, let $\tilde{\phi}$ be the map on $\bhp$ defined by 
$\tilde{\phi}(x)=\sum_{j\in\bj}\pi(a_j^*)x\pi(a_j)$  and set $\tr=\tom\circ\tilde\phi$. Then $\tilde{\phi}(x)=x$ for each $x\in\pi(\crr)^{\prime}$, hence $\tr|\pi(\crr)^{\prime}=\tom|\pi(\crr)^{\prime}$. Moreover, $\tr$ extends $\rho\circ\pi^{-1}$.

(ii)$\Rightarrow$(iii) Take for $\pi$ a faithful normal representation on a Hilbert space $\h$ such that $\omega$ is the restriction of a vector state  $\tom$  on $\bh$. (For example, $\crr$ may be in the standard form \cite[Chapter IX]{T} so that all normal states on $\crr$ and $\crr^{\prime}$ are vector states.) For simplicity of notation, we may assume that $\crr\subseteq\bh$, that is, $\pi={\rm id}$. Then with $\tr$  as in (ii), we have  $\tr|\crr=\rho$ and $\tr|\crr^{\prime}=\tom|\crr^{\prime}$. 

(iii)$\Rightarrow$(i) Assume that $\crr$ is represented faithfully on a Hilbert space $\h$ such that $\omega$ is the restriction of a vector state $\tom$ on $\bh$ and that  $\tr$ is a normal state on $\bh$ such that  $\tr|\crr=\rho$ and $\tr|\crr^{\prime}=\tom|\crr^{\prime}$. Let $\xi\in\h$ be such that $\tom(x)=\langle x\xi,\xi\rangle$ ($x\in\bh$). As a normal state, $\tr$ is of the form
$$\tr(x)=\langle x^{(\infty)}\eta,\eta\rangle\ \ \ (x\in\bh),$$
where $x^{(\infty)}$ denotes the direct sum of countably many copies of $x$ acting on the direct sum $\h^{\infty}$ of countably many copies of $\h$ and $\eta\in\h^{\infty}$.
Now from $\tom(x)=\tr(x)$ for all $x\in\crr^{\prime}$ we have
$$\langle x\xi,\xi\rangle=\langle x^{(\infty)}\eta,\eta\rangle\ \ \ (x\in(\crr^{\prime})).$$
Replacing $x$ by $x^*x$, it follows that there exists an isometry $u:[\crr^{\prime}\xi]\to[(\crr^{\prime})^{(\infty)}\eta]$ such that $u\xi=\eta$ and $uy=y^{(\infty)}u$ for all $y\in\crr^{\prime}$.
This $u$ can be extended to a partial isometry from $\h$ into $\h^{\infty}$, denoted again by $u$, by declaring it to be $0$ on the orthogonal complement of  $[\crr^{\prime}\xi]$ in $\h$. Then $u$ intertwines the identity representation ${\rm id}$ of  $\crr^{\prime}$ and the representation ${\rm id}^{\infty}$ and is therefore a column $(u_j)$, where $u_j\in\crr$.  For $r\in\crr$ we have 
$$\rho(r)=\tr(r)=\langle r^{(\infty)}\eta,\eta\rangle=\langle r^{(\infty)}u\xi,u\xi\rangle=\langle u^*r^{(\infty)}u\xi,\xi\rangle=\omega(u^*r^{(\infty)}u).$$
Thus $\rho=\omega\circ\psi$, where $\psi$ is a map on $\crr$, defined by $\psi(r)=u^*r^{(\infty)}u=\sum_ju_j^*ru_j$. This map $\psi$ is not necessarily unital, but from $$\omega(1)=1=\rho(1)=\omega(\psi(1))=\omega(\sum_ju_j^*u_j)=\omega(u^*u)\ \ \mbox{and}\ \ u^*u\leq1$$ we infer that $1-u^*u\leq 1-p$, where $p$ is the support projection of $\omega$. Hence $p\leq u^*u$ and we may replace $\psi$ by the unital map $\phi$ defined by
$\phi(r)=p\psi(r)p+(1-p)r(1-p)$, which satisfies $\omega\circ\phi=\omega\circ\psi=\rho$ and has the required form:
$$\psi(r)=\sum_jpu_j^*ru_jp+p^{\perp}rp^{\perp},\  \mbox{where}\  \sum_jpu_j^*u_jp+p^{\perp}p^{\perp}=pu^*up+p^{\perp}=1.$$

The equivalence (i)$\Leftrightarrow$(iv) is proved by similar arguments and we will omit the details, just note that $\crr\cong\pi_{\omega}(\crr)\oplus\ker\pi_{\omega}$. \end{proof}

\section{The case of C$^*$-algebras}

In a general C$^*$-algebra $A$ there is usually not enough module homomorphisms of $A$ into its center $Z$  and even if $Z=\bc1$ there can be many ideals in $A$. Functionals on $A$ usually do not preserve ideals, hence can not be  approximated by elementary operators. Therefore we will use for general C$^*$-algebras a  different approach from that in the previous section, not trying to construct an explicit map sending one state to another. For C$^*$-algebras with Hausdorff primitive spectrum the situation nevertheless resembles the one for von Neumann algebras. 

\begin{theorem}\label{H}Let $\omega,\rho$ be hermitian linear functionals on a C$^*$-algebra $A$ with Hausdorff primitive spectrum $\check{A}$ and center $Z$. Then $\rho$ is in the weak* closure $\overline{\omega\circ\ea}$ of the set $\omega\circ\ea$ if and only if the following condition is satisfied:

\smallskip
\noindent(A)\ \  \ \ $\rho|Z=\omega|Z$ and $\|c\rho\|\leq\|c\omega\|$ for each $c\in Z_+$.
\end{theorem}

\begin{proof}To prove the nontrivial direction of the theorem, suppose that the condition (A) is satisfied, but that $\rho\notin\overline{\omega\circ\ea}$. Then by the Hahn-Banach theorem
there exist $h\in A_h$ and $\alpha,\delta\in\br$, $\delta>0$, such that 
\begin{equation}\label{20}\omega(\psi(h))\leq\alpha\ \forall\psi\in\ea\ \ \mbox{and}\ \ \rho(h)\geq\alpha+\delta.\end{equation}
Since $\rho|Z=\omega|Z$, in particular $\rho(1)=\omega(1)$, we may replace $h$ by $h+\gamma1$ (and $\alpha$ with $\alpha+\gamma\omega(1)$) for a sufficiently large $\gamma\in\br$ and thus assume that $h$ is positive in invertible. Given $\varepsilon>0$, let $a\in A_h$ be such that $$-1\leq a\leq1\ \ \mbox{and}\ \ \omega_+(a)-\omega_-(a)=\omega(a)>\|\omega\|-\varepsilon=\omega_+(1)+\omega_-(1)-\varepsilon.$$ 
By a well-known argument, which we now recall, this implies the relations (\ref{21}) below. Namely, from the above we have $\omega_+(1-a)<-\omega_-(1+a)+\varepsilon$ and (since $1-a\geq0$ and $1+a\geq0$) this implies that $\omega_+(1-a)<\varepsilon$ and $\omega_-(1+a)\leq\varepsilon$. Thus $\omega_+(a_+)\geq\omega_+(a)>\omega_+(1)-\varepsilon$ and $\omega_-(a_+)=\omega_-(1+a)-\omega_-(1-a_-)\leq\omega_-(1+a)\leq\varepsilon$. In conclusion \begin{equation}\label{21}\omega_+(1-a_+)<\varepsilon\ \ \mbox{and}\ \ \omega_-(a_+)\leq\varepsilon.\end{equation}
For each $t\in\check{A}$ let $m(t)$ and $M(t)$ be the smallest and the largest point in the spectrum $\sigma(h(t))$ of $h(t)\in A/t$. Since $\check{A}$ is Hausdorff by assumption, the two functions
$M$ and $m$ (given by $M(t)=\|h(t)\|$ and $m(t)=\|h(t)^{-1}\|^{-1}$) are continuous \cite[4.4.5]{Pe} and therefore define elements of the center $Z$ of $A$ by the Dauns-Hoffman theorem. Set
$$b=Ma_++m(1-a_+).$$
For each $t\in\check{A}$ the spectrum of $b(t)$ is $\sigma(m(t)1+(M(t)-m(t))a_+(t))$ and is contained in $m(t)+(M(t)-m(t))[0,1]\subseteq[m(t),M(t)]$ since $\sigma(a_+(t))\subseteq\sigma(a_+)\subseteq[0,1]$.
Thus the numerical range $W(b(t))$ of $b(t)$ (which for normal elements coincides with the convex hull of the spectrum) is contained in $W(h(t))=[m(t),M(t)]$. Therefore by \cite[4.1]{M} $b$ is in the norm closure of the set $\{\psi(h):\, \psi\in\ea\}$, hence
\begin{equation}\label{22}\omega(b)\leq\alpha\end{equation}
by the first relation in (\ref{20}).
On the other hand, we can estimate $\omega(b)$ as
$$\omega(b)=\omega_+(Ma_+)-\omega_-(Ma_+)+\omega_+(m(1-a_+))-\omega_-(m(1-a_+))$$
$$=\omega_+(M)-\omega_-(m)-\omega_+((M-m)(1-a_+))-\omega_-((M-m)a_+)$$
$$\geq\omega_+(M)-\omega_-(m)-\|M-m\|(\omega_+(1-a_+)+\omega_-(a_+))$$ (\mbox{since} $0\leq (M-m)(1-a_+)\leq\|M-m\|(1-a_+)$ and $0\leq (M-m)a_+\leq\|M-m\|a_+$)
$$\geq\omega_+(M)-\omega_-(m)-2\|M-m\|\varepsilon\ \ \mbox{(by \ref{21})}.$$
Thus by (\ref{20}), (\ref{22}) and since $m\leq h\leq M$ implies that $\rho_+(h)\leq \rho_+(M)$ and $\rho_-(h)\geq \rho_-(m)$ we have now
$$\omega_+(M)-\omega_-(m)-2\varepsilon\|M-m\|+\delta\leq\rho(h)=\rho_+(h)-\rho_-(h)\leq\rho_+(M)-\rho_-(m).$$
This can be rewritten as
\begin{equation}\label{201}(\omega_+-\rho_+)(M)\leq(\omega_--\rho_-)(m)+2\varepsilon\|M-m\|-\delta\end{equation}
or (since $\omega|Z=\rho|Z$ implies that $(\omega_--\rho_-)|Z=(\omega_+-\rho_+)|Z$ and since $M,m\in Z$)
$$(\omega_+-\rho_+)(M-m)\leq2\varepsilon\|M-m\|-\delta.$$
Since this holds for all $\varepsilon>0$, by choosing small enough $\varepsilon$, it follows that $(\omega_+-\rho_+)(M-m)<0$.
But, $Z\ni M-m\geq0$ and $\omega_+|Z\geq\rho_+|Z$ by Lemma \ref{le00}, hence $(\omega_+-\rho_+)(M-m)\geq0$, which is a contradiction.
\end{proof}

The following corollary can be proved in the same way as Corollary \ref{co0}, so we will omit the proof.

\begin{co}\label{Hc}If $\omega$ and $\rho$ are states on a C$^*$-algebra $A$ with  Hausdorff primitive spectrum, then  $\rho\in\overline{\omega\circ\ea}$ if and only if $\rho|Z=\omega|Z$.
\end{co}

Before stating our main result in this section we need a lemma. Recall that a projection $p$ in the center of the universal von Neumann envelope $\crr$ of a C$^*$-algebra $A$ is called {\em open} if there is an ideal $J$ in $A$ such that $\overline{J}=p\crr$, where $\overline{J}$ is the weak* closure of $J$ in $\crr$.

\begin{lemma}\label{leh1}Let $\crr$ be the universal von Neumann envelope of a  C$^*$-algebra $A$ and $\cz$ the center of $\crr$. For each $h\in A_+$ the central carrier $C_h$ of $h$ in $\crr$ can be approximated in norm by linear combinations of open central projections in $\crr$, where the coefficients in each combination are positive.
\end{lemma}

\begin{proof}By definition the central carrier $z$ of $h$ is the infimum of all $c$ in $\cz$  such that 
$h\leq c$. If $\Delta$ is the maximal ideal space of $\cz$, then $z$ corresponds (via the Gelfand isomorphism) to the function $\Delta\ni t\mapsto\|h(t)\|$,
where $h(t)$ is the coset of $h$ in $\crr/t\crr$. (This function is continuous by \cite{G}.) Thus we will regard $z$ as a function on $\Delta$. Let $[m,M]$ be an interval containing the range of $z$, where $m\geq0$ and $M=\|h\|=\|z\|$.  Given $a\in A_+$, the set
$U=\{t\in\delta:\ a(t)\ne0\}$ is open since the function $\Delta\ni t\mapsto\|a(t)\|\in\br$ is continuous. The weak* closure of the ideal generated by $a$ in $\crr$ is of the form $p\crr$ for a unique projection $p\in\cz$ and $p$ is open by definition. Since the quotient algebras $\crr/t\crr$ have only scalars in their centers, $p(t)=1$ for each $t\in U$, hence also for each $t\in\overline{U}$ by continuity, so $p\geq q$, where $q\in\cz$ is the projection that corresponds to the characteristic function of $\overline{U}$. But from the definition of $U$ we see that $qa=a$ and this implies that  $qb=b$ for each $b$ in the ideal generated by $a$. Hence $qp=p$ and it follows that $q=p$. In particular, for each $r\in\br_+$ the projection that corresponds to the closure of the set $U_r=\{t\in\Delta:\, z(t)>r\}$ is open since $U_r$ is just the set $\{t\in\Delta:\, a(t)\ne0\}$, where $a=(h-r)_+$. (This has been observed already by Halpern in \cite[proof of Lemma 6]{Ha2}.)  
Given $\varepsilon>0$, for each $k\in\bn$ let $p_k$ be the  projection corresponding to the closure of the set $U_k=\{t\in\Delta:\, z(t)>M-k\varepsilon\}$. Then $0=p_0\leq p_1\leq p_2\leq\ldots\leq p_n=1$, where $n\in\bn$ is such that $M-n\varepsilon< m$ and $M-(n-1)\varepsilon\geq m$. Now from $1=(p_1-p_0)+(p_2-p_1)+\ldots+(p_n-p_{n-1})$ we have that $F_k:=\overline{U}_{k}\setminus\overline{U}_{k-1}$ are disjoint closed and open sets that cover $\Delta$ and for $t\in F_k$ we have that $M-k\varepsilon\leq z(t)\leq M-(k-1)\varepsilon$. Thus, if we choose in each interval $[M-k\varepsilon,M-(k-1)\varepsilon]$ a point $\lambda_k\geq0$ and set $c:=\sum_{k=1}^n\lambda_k(p_k-p_{k-1})$, it follows that $\|z-c\|\leq\varepsilon$.
Finally, observe that $$c=(\lambda_1-\lambda_2)p_1+(\lambda_2-\lambda_3)p_2+\ldots+(\lambda_{n-1}-\lambda_n)p_{n-1}+\lambda_np_n$$
is a linear combination with positive coefficients of open projections.
\end{proof}
 
The following theorem is a special case of Theorem \ref{H2} below, but it is used in the proof of that theorem. 
\begin{theorem}\label{H1} Let $\omega$ and $\rho$ be states on a  C$^*$-algebra $A$. Then $\rho$ is in the weak* closure $\overline{\omega\circ\ea}$ of the set $\omega\circ\ea=\{\omega\circ\psi:\, \psi\in\ea\}$, where $\ea$ is the set of all unital completely positive elementary complete contractions on $A$, if and only if 
$\|\rho|J\|\leq\|\omega|J\|$ for each   ideal $J$ of $A$. 
\end{theorem}

\begin{proof} Evidently $\rho\in\overline{\omega\circ\ea}$ implies that $\|\rho|J\|\leq\|\omega|J\|$ for each ideal $J$ in $A$ since maps in $\ea$ are contractive and preserve ideals. For the converse,  suppose that $\rho\notin\overline{\omega\circ\ea}$. Then by the Hahn-Banach theorem there exist  $h\in A_h$ and $\alpha\in\br$ such that (\ref{20}) holds, that is $\omega(\psi(h))\leq\alpha$ for all $\psi\in\ea$, while $\rho(h)>\alpha.$
Replacing $h$ by $h+\beta1$ for a sufficiently large $\beta\in\br$ (and consequently $\alpha$ by $\alpha+\beta$), we may assume that $h$ is positive.

Let $\crr$ be the universal von Neumann envelope of $A$ and denote the unique weak* continuous extensions of $\omega$ and $\rho$ to $\crr$ by the same two letters. We will use the same notation as in the proof of Lemma \ref{leh1}. Thus $z$ is the infimum of all $c$ in $\cz$  such that 
$h\leq c$. Since $W(z(t)1)=\{z(t)\}\subseteq W(h(t))$ for each $t\in\Delta$, it follows by \cite[3.3]{M} that $z\in\overline{{\rm co}_{\crr}}(h)$ (= the weak* closure of the $\crr$-convex hull of $h$), hence by the first relation in (\ref{20}) \begin{equation}\label{000}\omega(z)\leq\alpha,\end{equation} since each map $\psi$ of the form $x\mapsto\sum_ib_i^*xb_i$ ($b_i\in\crr$, $\sum_ib_i^*b_i=1$) can be approximated by maps of the form $x\mapsto\sum_ia_i^*xa_i$ ($a_i\in A$, $\sum_ia_i^*a_i=1$). (This follows by using the Kaplansky density theorem in  ${\rm M}_n(\crr)$ to approximate the column $b=(b_1,\ldots,b_n)^T$ by $(a_1,\ldots,a_n)^T$.) 
Since $\omega$ and $\rho$ are states,  the hypothesis $\|\rho|J\|\leq\|\omega|J\|$ for each ideal $J$ in $A$ means that $\rho(p)\leq\omega(p)$ for each open projection $p\in\cz$. Then it follows by Lemma \ref{leh1} that $\rho(z)\leq\omega(z)$. But from $h\leq z$ and using (\ref{000})
we have now that
$\rho(h)\leq\rho(z)\leq\omega(z)\leq\alpha$, which is in contradiction with the previously established relation $\rho(h)>\alpha$. 
\end{proof}

The naive attempt to generalize Theorem \ref{H1} to hermitian functionals fails, as shown by the following example.
The example also shows that the assumption in Theorem \ref{H}, that $A$ has Hausdorff primitive spectrum,
is not redundant and that in Theorem \ref{thvn} the normality of $\omega$ and $\rho$ is not redundant.
\begin{ex}\label{ex1}For a separable Hilbert space $\h$, let $\omega_1$ be a normal and $\omega_2$ a singular state on $\bh$,
$\rho_1$ and $\rho_2$ positive normal functionals on $\bh$ with orthogonal supports such that $\rho_1(1)=\frac{1}{2}=\rho_2(1)$. Set $\omega=\omega_1-\omega_2$ and $\rho=\rho_1-\rho_2$.
Then $\rho(1)=0=\omega(1)$, $\|\rho\|=\rho_1(1)+\rho_2(1)=1=\omega_1(1)\leq\|\omega\|$.  Since $\rho_1$, $\rho_2$ and $\omega_1$ are normal, while $\omega_2$ is singular (which means that  $\omega_2$ annihilates the  ideal $\kh$ of all compact operators on $\h$), we have $\|\rho|\kh\|=\|\rho\|=\rho_1(1)+\rho_2(1)=\omega_1(1)=\|\omega_1|\kh\|=\|\omega|\kh\|\leq\|\omega\|$. Thus $\|\rho|J\|\leq\|\omega|J\|$ for each ideal $J$ of $\bh$ and $\omega$ and $\rho$ agree on the center $\bc1$ of $\bh$. But nevertheless, 
$\rho\notin\overline{\omega\circ{\rm E}(\bh)}$ since on $\kh$ all elements of $\overline{\omega\circ{\rm E}(\bh)}$ acts as elements of
$\overline{\omega_1\circ{\rm E}(\bh)}|\kh$ and are therefore positive, while $\rho|\kh=(\rho_1-\rho_2)|\kh$ is not positive.
\end{ex}

To generalize Theorem \ref{H1} to hermitian functionals we need a lemma.

\begin{lemma}\label{leh}For each hermitian functional $\omega$ on a C$^*$-algebra $A$ we have 
$$\overline{\omega\circ\ea}=\overline{\omega_+\circ\ea}-\overline{\omega_-\circ\ea}.$$
\end{lemma}

\begin{proof}Suppose that $\rho\in\overline{\omega\circ\ea}$ and let $(\psi_k)$ be a net in $\ea$ such that $\rho(a)=\lim_k\omega(\psi_k(a))$ for all $a\in A$. Extend $\omega$, $\rho$ and each $\psi_k$ weak* continuously to the universal von Neumann envelope $\crr$ of $A$ and denote the extensions by the same symbols. Let $\psi$ be a weak* limit point of the net $(\psi_k)$ and note that $\psi$ is a unital completely positive (hence contractive) module map over the center $\cz$ of $\crr$. Set $\rho_1=\omega_+\circ\psi|A$ and $\rho_2=\omega_-\circ\psi|A$. Then $\rho_1\in\overline{\omega_+\circ\ea}$, $\rho_2\in\overline{\omega_-\circ\ea}$ and $\rho=\omega\circ\psi|A=\rho_1-\rho_2$. This proves the inclusion $\overline{\omega\circ\ea}\subseteq\overline{\omega_+\circ\ea}-\overline{\omega_-\circ\ea}$.

To prove the reverse inclusion, suppose that $\rho_1\in\overline{\omega_+\circ\ea}$ and $\rho_2\in\overline{\omega_-\circ\ea}$. Then there exist nets of maps $\phi_k$ and $\psi_k$ in $\ea$ such that $\rho_1=\lim_k\omega_+\circ\phi_k$ and $\rho_2=\lim_k\omega_-\circ\psi_k$. Let $p$ and $q$ be the support projections in $\crr$ of $\omega_+$ and $\omega_-$ (where $\omega_+$ and $\omega_-$ have been weak* continuously extended to $\crr$). Let $(a_n)$ be a net of positive contractions in $A$ strongly converging to $p$ in $\crr$, set $b_n=\sqrt{1-a_n^2}$ and define maps $\phi_{k,n}$ and $\psi_{k,n}$ on $A$ by
$$\phi_{k,n}(x)=a_n\phi_k(x)a_n\ \ \mbox{and}\ \ \psi_{k,n}(x)=b_n\psi_k(x)b_n.$$
The nets $(\omega_+(b_n^2))=(\omega_+(1-a_n^2))$ and $(\omega_-(a_n^2))= (\omega_-(1-b_n^2))$ all converge to $0$. From this we will verify below by using the Cauchy-Schwarz inequality for positive functionals that $\lim_{k,n}\omega_+\circ\phi_{k,n}=\rho_1$,
$\lim_{k,n}\omega_-\circ\psi_{k,n}=\rho_2$,
$\lim_{k,n}\omega_+\circ\psi_{k,n}=0$ and $\lim_{k,n}\omega_-\circ\phi_{k,n}=0$ pointwise on $A$, hence
$$\rho:=\rho_1-\rho_2=\lim_{k,n}[(\omega_+-\omega_-)\circ(\phi_{k,n}+\psi_{k,n})]=\lim_{k,n}\omega\circ\theta_{k,n},$$
where $\theta_{k,n}:=\phi_{k,n}+\psi_{k,n}$.
Evidently each $\theta_{k,n}$ is elementary completely positive map and also unital since $\theta_{k,n}(1)=a_n\phi_k(1)a_n+b_n\psi_k(1)b_n=a_n^2+b_n^2=1$. Thus $\rho\in\overline{\omega\circ\ea}$, verifying the inclusion $\overline{\omega\circ\ea}\supseteq\overline{\omega_+\circ\ea}-\overline{\omega_-\circ\ea}$.
Now we will verify that $\lim_{k,n}\omega_+\phi_{k,n}=\rho_1$, the verification of the other three limits that we have used is similar. For each $x\in A$ we estimate
$$|\rho_1(x)-\omega_+(\phi_{k,n}(x))|=|\rho_1(x)-\omega_+(a_n\phi_k(x)a_n)|$$$$\leq|\rho_1(x)-\omega_+(\phi_k(x))|+|\omega_+((1-a_n)\phi_k(x))|+|\omega_+(a_n\phi_k(x)(1-a_n))|$$$$
\leq|\rho_1(x)-\omega_+(\phi_k(x))|+\omega_+((1-a_n)^2)^{1/2}\omega_+(\phi_k(x)^*\phi_k(x))^{1/2}$$$$+\omega_+(\phi_k(x)^*a_n^2\phi_k(x))^{1/2}\omega_+((1-a_n)^2)^{1/2}$$
$$\leq|\rho_1(x)-\omega_+(\phi_k(x))|+2\omega_+((1-a_n)^2)^{1/2}\|\omega_+\|^{1/2}\|x\|.$$
Both terms in the last line of the above expression converge to $0$.
\end{proof}

\begin{theorem}\label{H2}Let $\omega$ and $\rho$ be hermitian functionals on a  C$^*$-algebra $A$. Then $\rho\in\overline{\omega\circ\ea}$ if and only if there exist positive functionals $\rho_1$ and $\rho_2$ on $A$ satisfying the following condition:

\smallskip
\noindent(B) $\rho=\rho_1-\rho_2$, $\rho_1(1)=\omega_+(1)$, $\rho_2(1)=\omega_-(1)$, $\|\rho_1|J\|\leq\|\omega_+|J\|$ and $\|\rho_2|J\|\leq\|\omega_-|J\|$ for all ideals $J$ in $A$. 

\smallskip
\noindent(In particular $\|\rho|J\|\leq\|\omega|J\|$.) If $\omega$ is positive, then the condition (B) simplifies to $\rho(1)=\omega(1)$ and $\|\rho|J\|\leq\|\omega|J\|$ for all ideals $J$.
\end{theorem}

\begin{proof}Suppose that $\rho\in\overline{\omega\circ\ea}$. Using the notation introduced in the first part of the proof of Lemma \ref{leh}, we have observed that the map $\psi$ on $\crr$ introduced in that proof is a contractive unital
$\cz$-bimodule map.  Thus for any ideal $J$ in $A$, if $p\in\cz$ is the projection satisfying $\overline{J}=p\crr$, then $\psi(\overline{J})=\psi(p\crr)=p\psi(\crr)\subseteq\overline{J}$. Since $\omega$ (and hence also $\omega_+$ and $\omega_-$) are
weak* continuous on $\crr$, we have $\|\omega_+|\overline{J}\|=\|\omega_+|J\|$ and $\|\omega_-|\overline{J}\|=\|\omega_-|J\|$.  With $\rho_1=\omega_+\circ\psi|A$ and $\rho_2=\omega_-\circ\psi|A$ (as in the proof of Lemma \ref{leh}) we have $\rho=\rho_1-\rho_2$, $\rho_1(1)=\omega_+(1)$, $\rho_2(1)=\omega_-(1)$,
$$\|\rho_1|J\|\leq\|\omega_+\circ\psi|\overline{J}\|\leq\|\omega_+|\overline{J}\|=\|\omega_+|J\|$$
and similarly $\|\rho_2|J\|\leq\|\omega_-|J\|$. Therefore also
$$\|\rho|J\|=\|\rho_1|J-\rho_2|J\|\leq\|\rho_1|J\|+\|\rho_2|J\|\leq\|\omega_+|J\|+\|\omega_-|J\|=\|\omega|J\|.$$

Conversely, assume the existence of positive functionals $\rho_1$ and $\rho_2$ on $A$ satisfying the norm inequalities in condition (B). Then by Theorem \ref{H1} $\rho_1\in\overline{\omega_+\circ\ea}$ and $\rho_2\in\overline{\omega_-\circ\ea}$, hence by Lemma \ref{leh} $\rho\in\overline{\omega\circ\ea}$. 
\end{proof}

\section{Maximally mixed states}

For functionals $\omega$ and $\rho$ on a C$^*$-algebra $A$ let us say that $\rho$  is more mixed than  $\omega$ if $\rho\in\overline{\omega\circ\ea}$ (where the bar denotes weak* closure).  Applying Zorn's lemma to the family of all weak* closed $\ea$-invariant subsets of $\overline{\omega\circ\ea}$ we see that in $\overline{\omega\circ\ea}$ there exist minimal $\ea$-invariant compact non-empty subsets, which are evidently of the form $\overline{\rho\circ\ea}$ for some $\rho$ and such  $\rho$  are called  maximally mixed. Thus a functional $\omega$ is {\em maximally mixed} if $\rho\in\overline{\omega\circ\ea}$ implies that $\omega\in\overline{\rho\circ\ea}$. If $A$ has Hausdorff primitive spectrum, Corollary \ref{Hc} implies that all states on $A$ are maximally mixed. 
The same conclusion holds  for  liminal C$^*$-algebras.

\begin{co}On a liminal C$^*$-algebra $A$ every state $\omega$ is maximally mixed.
\end{co}

\begin{proof}If $\rho\in\overline{\omega\circ\ea}$, then by Theorem \ref{H1} $\|\rho|J\|\leq\|\omega|J\|$ for each ideal $J$ in $A$. Denoting by $p$ the projection in $\crr:=A^{\sharp\sharp}$ such that $\overline{J}=p\crr$, this means that $\rho(p)\leq\omega(p)$ for each open central projection $p$, where $\omega$ and $\rho$ have been weak* continuously extended to $\crr$. Since $A$ is liminal, such projections are strongly dense in the set of all central projcections  by \cite{DH}, hence it follows that $\rho(p^{\perp})\leq\omega(p^{\perp})$. Since $\rho(p)+\rho(p^{\perp}=\rho(1)=1=\omega(1)=\omega(p)+\omega(p^{\perp})$, we conclude that $\rho(p)=\omega(p)$, that is $\|\rho|J\|=\|\omega|J\|$.
By Theorem \ref{H1} this implies that $\omega\in\overline{\rho\circ\ea}$.
\end{proof}

Perhaps the simplest C$^*$-algebras on which not all states are maximally mixed are C$^*$-algebras that have only one maximal ideal and this ideal is not $0$. 

\begin{ex}\label{lm}Suppose that a unital C$^*$-algebra $A$ has only one maximal ideal $M$ (for example, $A$ may be simple or a factor). Then a state $\omega$ on $A$ is maximally mixed if and only if $\omega|M=0$.

\begin{proof} Suppose that $\omega|M=0$ and let $\rho\in\overline{\omega\circ\ea}$. Then $\rho|M=0$, hence also $\rho(J)=0$ for each proper ideal $J$ of $A$ since $J\subseteq M$. Thus $\|\omega|J\|=\|\rho|J\|$ for each ideal $J$ of $A$, so $\omega\in\overline{\rho\circ\ea}$ by Theorem \ref{H1}.

Suppose now that $\omega|M\ne0$. Let $\rho$ be any state on $A$ such that $\rho|M=0$. Then $\|\rho|J\|\leq\|\omega|J\|$ for all ideals $J$, hence $\rho\in\overline{\omega\circ\ea}$ by Theorem \ref{H1}. But $\omega\notin\overline{\rho\circ\ea}$ since $\rho|M=0$ and $\omega|M\ne0$, thus $\omega$ is not maximally mixed. 
\end{proof}
\end{ex}

\begin{re}\label{req}If $K$ is an ideal of $A$, each state  $\omega$ on $A$ satisfying $\omega(K)=0$ may be regarded as a state on $A/K$, say $\dot\omega$. Note that  $\dot\omega$ is maximally mixed on $A/K$ if and only if $\omega$ is maximally mixed on $A$. Indeed, denoting by $q:A\to A/K$ the natural map, $q(J)$ is an ideal in $A/K$ for each ideal $J$ in $A$ and all ideals in $A/K$ are of such a form. Moreover,  $\|\omega|J\|=\|\dot\omega|q(J)\|$, hence the claim follows from Theorem \ref{H1}.
\end{re} 

Example \ref{lm} is generalized in Theorem \ref{t} below. The proof of Theorem \ref{t} is inspired by an idea from  \cite[3.10]{ART}, but we will avoid using a background  result from \cite{ART1}, that is used in \cite[3.10]{ART}, and present a short self-contained proof. Recall that the strong radical $J_A$ of $A$ is the intersection of all maximal ideals in $A$.

\begin{theorem}\label{t}(i) $\omega(J_A)=0$ for each maximally mixed state $\omega$ on $A$.

(ii) If a state $\omega$ on $A$ annihilates some intersection $M_1\cap M_2\cap\ldots\cap M_n$ of finitely many maximal ideals  in $A$, then $\omega$ is maximally mixed.

Thus the set $S_m(A)$ of maximally mixed states on $A$ is a weak* dense subset of $S(A/J_A)$ (= the set of states on $A$ that annihilate $J_A$).
\end{theorem}

\begin{proof}(i) Let $D=S(A/J_A)$ and $\omega$ a maximally mixed state on $A$.  Suppose that $\omega\notin D$. Then $\overline{\omega\circ\ea}\cap D=\emptyset$, otherwise this intersection would be a weak* closed proper
$\ea$-invariant subset of $\overline{\omega\circ\ea}$, which would contradict the fact that $\omega$ is maximally mixed.  Thus by the Hahn-Banach theorem there exist $\alpha,\beta\in\br$ and $h\in A_h$ such that
\begin{equation}\label{m1}\rho(h)\leq\alpha\ \forall\rho\in D\ \ \mbox{and}\ \ \omega(\psi(h))\geq\beta>\alpha\ \forall\psi\in\ea.\end{equation} Replacing $h$ by $h+\gamma1$ for a sufficiently large $\gamma\in\br_+$ (and modifying $\alpha, \beta$), we may assume that $h$ is positive. Then the first relation in (\ref{m1}) means that $\|\dot{h}\|\leq\alpha$, where $\dot{h}$ denotes the coset of $h$ in $A/J_A$. The (algebraic) numerical range $W_{A/J_A}(h)$ of $\dot{h}$ is an interval, say $[c,d]$, contained in the numerical range $W_A(h)$ of $h$, which is an interval, say $[a,b]$; note that $a\leq c\leq d=\|\dot{h}\|\leq b=\|h\|$. Let $f:[a,b]\to[c,d]$ be the function, which act as the identity on $[c,d]$, and maps $[a,c]$ into $\{c\}$ and $[d,b]$ into $\{d\}$. For every proper ideal $K$ in $A$ the quotient $A/(K+J_A)$ is non-zero, for $K$ is contained in a maximal ideal $M$ and hence $K+J_A\subseteq M+J_A=M\ne A$. Since $W_{A/(K+J_A)}(h)\subseteq W_{A/K}(h)\cap W_{A/J_A}(h)$, this intersection is not empty, hence the interval $W_{A/K}(h)$ intersects $[c,d]$ and is therefore mapped by $f$ into itself. The numerical range $W_{A/K}(f(h))$ of the coset of $f(h)$ in $A/K$ is just the convex hull of the spectrum $\sigma_{A/K}(f(h))=f(\sigma_{A/K}(h)$, hence $W_{A/K}(f(h))\subseteq f(W_{A/K}(h))\subseteq W_{A/K}(h)$. This inclusion implies that $f(h)\in\overline{\ea(h)}$ by \cite{M}, hence $\omega(f(h))>\alpha$ by the second relation in (\ref{m1}). Since $\omega$ is a state, it follows that $W_A(f(h))$ intersects $(\alpha,\infty)$. But this is a contradiction since $W_A(f(h))$ is the convex hull of the spectrum $\sigma_A(f(h))=f(\sigma_A(h))\subseteq[c,d]=[c,\|\dot{h}\|]\subseteq[c,\alpha]$. Thus $\omega\in D$.

(ii) By the Chinese remainder theorem \cite[6.3]{Gr} there is a natural isomorphism $A/\cap_{j=1}^nM_j\cong\oplus_{j=1}^nA/M_j$, thus we may regard $\omega$ as a state on $\oplus_{j=1}^nA/M_j$. Since the algebras $A/M_j$ are simple, all states on them are maximally mixed by Example \ref{lm}. The same then holds for their direct sum, so all states on $A/\cap_{j=1}^nM_j$ are maximally mixed and (ii) follows by  Remark \ref{req}.

The set of all states that annihilate some finite intersection of maximal ideals of $A$ is convex and norming for $A/J_A$ (since the natural map $A/J_A\to\oplus_MA/M$, where the sum is over all maximal ideals in $A$, is a monomorphism, thus  isometric), hence weak* dense in $S(A/J_A)$  \cite[4.3.9]{KR}.
\end{proof}

\begin{re}A similar argument as in \cite[3.2]{ART} shows that the set $S_m(A)$ of all maximally mixed states on a C$^*$-algebra $A$ is always norm closed.\end{re} 

Recall that a C$^*$-algebra $A$ is {\em weakly central} if different maximal ideals of $A$ have different intersection with the center $Z$ of $A$. 

\begin{theorem}\label{L}If the set $S_m(A)$ of all maximally mixed states is weak* closed (which by Theorem \ref{t} just means that $S_m(A)=S(A/J_A)$), then each primitive ideal of $A$ containing $J_A$ is maximal. If $A$ is weakly central, then the converse also holds: if each primitive ideal containing $J_A$ is maximal, then $S_m(A)=S(A/J_A)$.
\end{theorem}

\begin{proof}By Remark \ref{req} a state $\omega$ on  $A/J_A$ is maximally mixed if and only if it is maximally mixed on $A$.  By \cite[3.10]{AG} the quotients of weakly central C$^*$-algebras are weakly central, so in particular $A/J_A$ is weakly central.  In this way we reduce the proof to the algebra $A/J_A$ (instead of $A$), which has strong radical $0$. Thus we may assume that $J_A=0$.

Suppose now that $S_m(A)=S(A)$. Then $S_m(A/P)=S(A/P)$ for each primitive ideal $P$ of $A$ by Remark \ref{req}. If $M$ is a maximal ideal of $A$ containing $P$, then $A/M$ is a quotient of $A/P$, hence each state $\rho\in S(A/M)$ can be regarded as a state on $A/P$ and therefore can be weak* approximated by convex combinations of vector states on $A/P$ (where $A/P$ has been faithfully represented on a Hilbert space). Since $A/P$ is primitive, as a consequence of the Kadison transitivity theorem each vector state is of the form $x\mapsto \theta(u^*xu)$ for a fixed state $\theta$ on $A/P$ with $\theta(M/P)\ne0$, where $u\in A/P$ is unitary \cite[5.4.5]{KR}. Thus $\rho\in\overline{\theta\circ{\rm E}(A/P)}$. But $\rho(M/P)=0$, while $\theta(M/P)\ne0$ if $M\ne P$, hence $\theta\notin\overline{\rho\circ{\rm E}(A/P)}$ if $M\ne P$. Thus $\rho$ can not be maximally mixed (on $A/P$ and hence also on $A$) if $P$ is not maximal.
This argument, which we have found in \cite[proof of 3.15]{ART}, shows that in general the equality $S_m(A)=S(A)$ can hold only if  all primitive ideals containing $J_A$ are maximal. If $A$ is weakly central and by our reduction above $J_A=0$, then the assumption that all primitive ideals are maximal implies that the primitive spectrum $\check{A}$ of $A$ is homeomorphic to the maximal ideal space $\Delta$ of $Z$
(via the map $\check{A}\ni M\mapsto M\cap Z\in\Delta$). Thus $\check{A}$ is Hausdorff and in this case Corollary \ref{Hc} shows that all states on $A$ are maximally mixed.
\end{proof}

It is well-known that each W$^*$-algebra $\crr$ is weakly central. If $\crr$ is properly infinite,  each primitive ideal $P$ containing $J_{\crr}$ is maximal. (Namely, by  by \cite[2.3]{Ha3} or \cite[8.7.21]{KR} the ideal $M:=P+J_{\crr}\supseteq\crr(P\cap\cz)+J_{\crr}$ is maximal, and $M=P$ if $P\supseteq J_{\crr}$.) So we can state the following corollary.

\begin{co}In a properly infinite von Neumann algebra $\crr$ maximally mixed states are just the  states that annihilate the strong radical $J_{\crr}$.
\end{co}

If $\crr$ is finite, primitive ideals are not necessarily maximal. (By \cite[4.7]{Ha} any ideal  $\crr t$, where $t$ is a maximal ideal of the center of $\crr$, is primitive, while using the central trace one can show that not all such ideals are maximal in $\crr=\oplus_{n}{\rm M}_n(\bc)$, for example.) Thus  the set of maximally mixed states on $\crr$ is not  weak* closed. 

{\em Throughout the rest of the paper $\crr$ is a W$^*$-algebra, $\cz$ its center and $\Delta$ the maximal ideal space of $\cz$. For each $t\in\Delta$ let $M_t$ be the unique maximal ideal of $\crr$ that contains $t$ \cite[8.7.15]{KR}). Note that $\phi(\crr t)=\phi(\crr)t\subseteq t$ for each  $\cz$-module map $\phi:\crr\to\cz$.} 

To prove that tracial states are maximally mixed, we need a lemma.

\begin{lemma}\label{mod4}A bounded $\cz$-module map
$\phi:\crr\to\cz\subseteq\crr$ preserves all  ideals of $\crr$ if and only if $\phi(M_t)\subseteq t$ for each $t\in\Delta$. If $\crr$ is properly infinite, this is equivalent to $\phi(J_{\crr})=0$. 
\end{lemma}

\begin{proof}Let $J$ be an
ideal in $\crr$ and $K=J\cap\cz$. As an ideal in $\cz$, $K$ can be identified with the set of all continuous functions on $\Delta$ than vanish on some closed subset $\Delta_K$ of $\Delta$, hence $K$ is the intersection of a family $\{t: t\in \Delta_K\}$ of maximal ideals of $\cz$. By \cite[8.7.15]{KR} there exists the largest ideal $J(K)$ in $\crr$ such that $J(K)\cap\cz=K$, and it follows from \cite[8.7.16]{KR} that $J(K)=\cap_{t\in\Delta_K}M_t$. Now $J\cap\cz=K$ implies that $J\subseteq J(K)$. Thus, if $\phi$ has the property that $\phi(M_t)\subseteq t$ for all $t\in\Delta$, then $\phi(J)\subseteq\phi(J(K))\subseteq\cap_{t\in\Delta_K}\phi(M_t)\subseteq\cap_{t\in\Delta_K}t=K\subseteq J$. 

If $\crr$ is properly infinite, then $M_t=\crr t+J_{\crr}$ for each $t\in\Delta$ by \cite [8.7.21 (1)]{KR}. Thus if $\phi(J_{\crr})=0$, then we have $\phi(M_t)=\phi(\crr)t\subseteq t$ for all $t\in\Delta$. Conversely, if $\phi(M_t)\subseteq t$ for all $t$, then $\phi(J_{\crr})=\phi(\cap_{t\in\Delta}M_t)\subseteq\cap_{t\in\Delta}t=0$.
\end{proof}

\begin{co}\label{mod6}A unital positive $\cz$-module map $\phi:\crr\to\cz\subseteq\crr$ is in the point-norm closure of elementary such maps (that is, $\phi\in\overline{\er}^{\rm{p.n.}}$) if and only if $\phi(M_t)\subseteq t$ for each $t\in\Delta$.
\end{co}

\begin{proof}By \cite[2.2]{M2} and  \cite[2.1]{M3} each completely contractive map $\phi:\crr\to\cz\subseteq\crr$ which preserves all ideals of $\crr$ is in the point-norm closure of maps of the form $x\mapsto a^*xb=\sum_{j=1}^na_j^*xb_j$, where $n\in\bn$, $a_j,b_j\in\crr$, $a:=(a_1,\ldots,a_n)^T$, $b:=(b_1,\ldots,b_n)$,  $\|a\|\leq1$ and $\|b\|\leq1$. If $\phi$ is unital, then we can modify such maps to unital   maps in the same way as in the proof of Theorem \ref{th}, which shows that $\phi\in\overline{\er}^{\rm p.n.}$.
\end{proof}

\begin{co}\label{vn}Let $\omega$ be a state of the form $\omega=\mu\circ\phi$,
where $\mu=\omega|\cz$  and $\phi:\crr\to\cz$ is a unital positive  $\cz$-module map. If $\phi(M_t)\subseteq t$ for each $t\in\Delta$, then $\omega$ is maximally mixed. In particular tracial states are maximally mixed.
\end{co}

\begin{proof} Suppose that $\rho\in\overline{\omega\circ\er}$. Then $\rho|\cz=\omega|\cz=\mu$, hence $$\omega=\mu\circ\phi=(\rho|\cz)\circ\phi=\rho\circ\phi.$$  By Corollary \ref{mod6} $\phi$ can be approximated in the point-norm topology by a net of  maps $\phi_k\in\overline{\er}^{\rm p.n.}$. Then $\omega(x)=\lim_k(\rho(\phi_k(x)))$ for all $x\in\crr$. This shows that $\overline{\omega\in\rho\circ\er}$, so $\omega$ is maximally mixed. 

Any tracial state $\omega$ annihilates the properly infinite part of $\crr$, hence we assume that $\crr$ is finite. Then $\omega=(\omega|\cz)\circ\tau$, where $\tau$ is the central trace on $\crr$ \cite[8.3.10]{KR}.  Since $M_t$ is of the form  $M_t=\{a\in\crr:\, \tau(a^*a)\in t\}$ by \cite[8.7.17]{KR}, for $a\in M_t$ we have by the Schwarz inequality  
$\tau(a)^*\tau(a)\leq\tau(a^*a)\in t$. This implies that $\tau(a)\in t$. Thus $\tau(M_t)\subseteq t$, hence $\omega$ is maximally mixed by the first part of the corollary.
\end{proof}

Are all maximally mixed states on  W$^*$-algebras of the form specified in Corollary \ref{vn}? Not quite. To investigate this, we still need some preparation.

\begin{lemma}\label{mod}{\em For each state $\omega$ on $\crr$ there exists a positive $\cz$-module map $\phi:\crr\to\cz$ such that $\omega=(\omega|\cz)\circ\phi$ and $p:=\phi(1)$ is a projection with $\omega(p)=1$.} 
\end{lemma}

\begin{proof}Let $\Phi$ be the universal representation of $\crr$, so that $\crr^{\sharp\sharp}$ is the weak* closure of $\Phi(\crr)$. Then the $*$-homomorphism $\Phi^{-1}:\Phi(\crr)\to\crr$ can be weak* continuously extended to a $*$-homomorphism $\Psi:\crr^{\sharp\sharp}\to\crr$; set $\tilde{\omega}=\omega\circ\Psi$ \cite[10.1.1, 10.1.12]{KR}. Let $\tilde{\cz}$ be the center of $\crr^{\sharp\sharp}$. Since $\tilde{\omega}$ is weak* continuous, by \cite{Ha} or \cite[1.4]{SZ} there exists a unique $\tilde{\cz}$-module homomorphism $\psi:\crr^{\sharp\sharp}\to\tilde{\cz}$ such that
$\tilde{\omega}=(\tilde{\omega}|\tilde{\cz})\circ\psi$ and $\psi(1)$ is the support projection $q$ of $\tilde{\omega}|\tilde{\cz}$. It is not hard to verify that $\phi:=(\Psi|\tilde{\cz})\circ \psi\circ \Phi$ has the properties stated in the lemma. 
\end{proof}

Let $\omega$ be a state on $\crr$, $\mu=\omega|\cz$ and let $\phi$, $p$ be as in Lemma \ref{mod}, so that $\omega=\mu\circ\phi$. Let $J$ be an ideal of $\crr$ and $K=J\cap\cz$. Let $(e_k)$ and $(f_l)$ be approximate units in $J$ and $K$ (respectively). Then
\begin{equation}\label{200}\|\omega|K\|=\|\mu|K\|=\lim_l\mu(f_l)\ \ \mbox{and}\ \ \|\omega|J\|=\lim_k\mu(\phi(e_k)).\end{equation}
We may regard $(f_l)$ and $(\phi(e_k))$ as two bounded increasing nets in the positive part of the unit ball of  $C(\Delta)$ ($\cong\cz$), hence they converge pointwise to some lower semi-continuous functions $f$ and $g$ (respectively) on $\Delta$. The  ideal $K$ of $C(\Delta)$ is of the form $K=\{a\in C(\Delta):\, a|\Delta_K^c=0\}$ for some open subset $\Delta_K$ of $\Delta$ and since $(f_l)$ is an approximate unit for $K$, it follows that $f$ is just the indicator function $\chi_{\Delta_K}$ of $\Delta_K$. 
Let $\Delta_p$ be the clopen subset of $\Delta$ that correspond to the projection $p=\phi(1)$ (that is, $p=\chi_{\Delta_p}$, the indicator function of $\Delta_p$).
Since $f_l\in J$ and $(e_k)$ is an approximate unit for $J$, $\lim_ke_kf_l=f_l$, hence $gf_l=\lim_k\phi(e_k)f_l=\lim_k\phi(e_kf_l)=\phi(f_l)=f_l\phi(1)=f_lp$ and $gf=\lim_lgf_l=\lim_lf_lp=fp$, that is $(g-\chi_{\Delta_p})\chi_{\Delta_K}=0$. This means that
\begin{equation}\label{203}g(t)=1\ \ \forall t\in\Delta_p\cap\Delta_K.\end{equation}

Since $(e_k)$ is an approximate unit,  for any $k_1$ and $k_2$ there exists $k_3\geq k_1,k_2$ so that $e_{k_3}\geq e_{k_1}$ and $e_{k_3}\geq e_{k_2}$, and  $(f_l)$ has the analogous property. Thus $f=\sup_lf_l$, $g=\sup_k\phi(e_k)$ and we may apply the version of the monotone convergence theorem for nets \cite[7.12]{Fo}. Thus, denoting by $\hat{\mu}$ the Radon measure on $\Delta$ that corresponds to $\mu$, we have $\lim_l\mu(f_l)=\sup_l\mu(f_l)=\sup_l\int_{\Delta}f_l\,d\hat{\mu}=\int_{\Delta}\sup_l f_l\,d\hat\mu=\int_{\Delta}f\, d\hat{\mu}=\hat{\mu}(f)$ and similarly $\lim_k\mu(\phi(e_k))=\hat{\mu}(g)$. Therefore by (\ref{200}) the equality $\|\omega|J\|=\|\omega|K\|$ is equivalent to $\hat\mu(g)=\hat\mu(f)=\hat\mu(\Delta_K)$.  By (\ref{203}) this condition $\hat\mu(g)=\hat\mu(f))$ means that $0=\hat\mu(g-f)=\int_{\Delta_K^c\cup\Delta_p^c}(g-f)\,d\hat\mu=\int_{\Delta_K^c}(g-\chi_{\Delta_K})\,d\hat\mu=\int_{\Delta_K^c}g\, d\hat\mu$,  since $\hat\mu(\Delta_p^c)=0$ (because $\mu(p)=1$).  As $g\geq0$, we conclude that $\|\omega|J\|=\|\omega|K\|$ if and only if $g(t)=0$ for $\hat\mu$-almost all $t\in\Delta_K^c$. Since   $(e_k)$ is an approximate unit of $J$, $\phi(e_k)(t)>0$ for some $k$ if and only if $\phi(a)(t)\ne0$ for some $a\in J$. Hence, since $g=\sup_k\phi(e_k)$, 
$$\{t\in\Delta_K^c:\, g(t)>0\}=\cup_k\{t\in\Delta_k^c:\, \phi(e_k)(t)>0\}=\cup_{a\in J}\{t\in\Delta_K^c:\, \phi(a)(t)\ne0\}.$$
This proves the following lemma. (Note that $g$ is lower semi-continuous, hence the set $\Delta_{\phi(J)|\Delta_K^c\ne0}$ in the lemma is $\hat\mu$-measurable.)

\begin{lemma}\label{lef}$\|\omega|J\|=\|\omega|(J\cap\cz)\|$ if and only if $\hat\mu(\Delta_{\phi(J)|\Delta_K^c\ne0})=0$, where
\begin{equation}\label{204}\Delta_{\phi(J)|\Delta_K^c\ne0}=\bigcup_{a\in J}\{t\in\Delta_K^c:\, \phi(a)(t)\ne0\}.\end{equation}
Here $K=J\cap\cz$ and $\Delta_K^c$ is the set of all common zeros of elements of $K$.
\end{lemma}
The following theorem says that maximally mixed states are those for which the corresponding $\phi $ almost (with respect to $\hat\mu$) preserve ideals.
\begin{theorem}Let $\omega$ be any state on $\crr$. Let $\omega=\mu\circ\phi$, where $\mu=\omega|\cz$ and $\phi:\crr\to\cz$ is a positive $\cz$-module map with $\phi(1)$ a projection. Denote by $\hat\mu$ the Radon measure on $\Delta$ that corresponds to $\mu$. Then $\omega$ is maximally mixed if and only if $\hat\mu(\Delta_{\phi(J)|\Delta_K^c\ne0})=0$ for each  ideal $J$ in $\crr$, where $K=J\cap\cz$, $\Delta_K^c=\{t\in\Delta:\, K\subseteq t\}$  and $\Delta_{\phi(J)|\Delta_K^c\ne0}$ is the set defined in (\ref{204}). 
\end{theorem}

\begin{proof}Suppose that $\rho\in\overline{\omega\circ\er}$. Then $\rho|\cz=\omega|\cz$ and by Theorem \ref{H1} $\|\rho|J\|\leq\|\omega|J\|$ for each ideal $J$ of $\crr$. If $\hat\mu(\Delta_{\phi(J)|\Delta_K^c\ne0})=0$ for each $J$, then by Lemma \ref{lef} $\|\omega|J\|=\|\omega|(J\cap\cz)\|$ for each $J$, hence
$\|\omega|J\|=\|\omega|(J\cap\cz)\|=\|\rho|(J\cap \cz)\|\leq\|\rho|J\|.$
Therefore by Theorem \ref{H1} $\omega\in\overline{\rho\circ\er}$, which proves that $\omega$ is maximally mixed. 

Conversely, if $\hat\mu(\Delta_{\phi(J_0)|\Delta_K^c\ne0})>0$ for some ideal $J_0$, then by Lemma \ref{lef} $\|\omega|(J_0\cap\cz)\|<\|\omega|J_0\|$. Let $\psi:\crr\to\cz$ be any positive unital $\cz$-module map that preserves ideals. (For example, the central trace, if $\crr$ is finite, as we have seen in the proof of Corollary \ref{vn}. If $\crr$ is properly infinite, preservation of ideals is equivalent to $\psi(J_{\crr})=0$ by Lemma \ref{mod4}, so we can take for $\psi$ the composition $\crr\stackrel{\eta}{\to}\crr/J_{\crr}\stackrel{\iota}{\to}\cz$, where $\eta$ is the natural map and $\iota$ is an extension of the inclusion $\cz\to\crr/J_{\crr}$.  Here $\cz$ is regarded as contained in $\crr/J_{\crr}$ since $\cz\cap J_{\crr}=0$, and $\iota$ exists by the C$^*$-injectivity of $\cz$.) Let $\rho=\mu\circ\psi$. Since $\psi(J)\subseteq J\cap\cz$ for each $J$, the set $\Delta_{\psi(J)|\Delta_{J\cap\cz}^c\ne0}$  is empty, hence by Lemma \ref{lef} $\|\rho|J\|=\|\rho|(J\cap\cz)\|$. Since $\rho|\cz=\mu=\omega|\cz$, we have $\|\rho|J\|=\|\rho|(J\cap\cz)\|=\|\omega|(J\cap\cz)\|\leq\|\omega|J\|$ for all $J$, hence $\rho\in\overline{\omega\circ\er}$ by Theorem \ref{H1}. But $\|\omega|J_0\|>\|\omega|(J_0\cap\cz)\|=\|\rho|(J_0\cap\cz)\|=\|\rho|J_0\|$ implies that $\omega\notin\overline{\rho\circ\er}$. Hence $\omega$ is not maximally mixed.
\end{proof}

\end{document}